\documentclass[12pt,a4paper]{article}
\usepackage[T1]{fontenc}
\usepackage{lmodern}
\usepackage[utf8]{inputenc}
\usepackage{amsthm}
\usepackage{amsmath,amssymb}
\usepackage{graphics}
\usepackage{todonotes}
\usepackage{lscape,graphicx}
\usepackage{enumitem}
\usepackage{amsfonts}
\usepackage[T1]{fontenc}
\usepackage{longtable}
\usepackage{authblk}
\usepackage{lipsum}
\usepackage{epsfig}
\usepackage{float}
\usepackage{hyperref}
\usepackage{comment}
\usepackage[normalem]{ulem}
\usepackage{multicol}
\usepackage{fancyhdr}
\usepackage[english]{babel}
\usepackage{mathrsfs}
\usepackage{tocloft}
\usepackage{xcolor}
\usepackage{titletoc}
\usepackage{mathtools}
\usepackage[toc,page]{appendix}

\addto{\captionsenglish}{}
\makeatletter
\def\thm@space@setup{%
  \thm@preskip=1cm plus .5cm minus .5cm
  \thm@postskip=.5cm plus .6cm minus .5cm 
}

\makeatother

\newtheorem{thm}{Theorem}

\newtheorem{lma}{Lemma}

\newtheorem{cor}{Corollary}

\newtheorem{rmk}{Remark}

\numberwithin{thm}{section}
\numberwithin{lma}{section}
\numberwithin{dfn}{section}
\numberwithin{cor}{section}
\numberwithin{rmk}{section}
\numberwithin{prop}{section}

\newcommand*{\thmref}[1]{Theorem~\ref{#1}}

\newcommand*{\lmaref}[1]{Lemma~\ref{#1}}

\title{Distribution of $\omega(n)$ over $h$-free and $h$-full numbers}

\date{}

\begin{document}

\author{Sourabhashis Das, Wentang Kuo, Yu-Ru Liu}

\newcommand{\Addresses}{{
  \bigskip
  \footnotesize

  Sourabhashis Das (Corresponding author), Department of Pure Mathematics, University of Waterloo, 200 University Avenue West, Waterloo, Ontario, Canada, N2L 3G1. \\
  Email address: \texttt{s57das@uwaterloo.ca}

  \medskip

  Wentang Kuo, Department of Pure Mathematics, University of Waterloo, 200 University Avenue West, Waterloo, Ontario, Canada, N2L 3G1. \\
  Email address: \texttt{wtkuo@uwaterloo.ca}
  
  \medskip

  Yu-Ru Liu, Department of Pure Mathematics, University of Waterloo, 200 University Avenue West, Waterloo, Ontario, Canada, N2L 3G1. \\
  Email address: \texttt{yrliu@uwaterloo.ca}

}}

%
%
%

\maketitle 

\begin{abstract}
Let $\omega(n)$ denote the number of distinct prime factors of a natural number $n$. In \cite{hardyram}, Hardy and Ramanujan proved that $\omega(n)$ has normal order $\log \log n$ over naturals. In this work, we establish the first and the second moments of $\omega(n)$ over $h$-free and $h$-full numbers using a new counting argument and prove that $\omega(n)$ has normal order $\log \log n$ over these subsets.
\end{abstract}

\section{Introduction}

For\footnotetext{\textbf{2020 Mathematics Subject Classification: 11N37, 11N05, 11N56.}} a\footnotetext{\textbf{Keywords: prime factors, prime-counting functions, normal order.}} natural\footnotetext{The research of W. Kuo and Y.-R. Liu are supported by NSERC discovery grants.} number $n$, let the prime factorization of $n$ be given as
\begin{equation}\label{factorization}
n = p_1^{s_1} \cdots p_r^{s_r},
\end{equation}
where $p_i$'s are its distinct prime factors and $s_i$'s are their respective multiplicities. Let $\Omega(n)$ denote the total number of prime factors in the factorization of $n$ and $\omega(n)$ denote the total number of distinct prime factors in the factorization of $n$. Hence $\Omega(n) = \sum_{i=1}^r s_i$ and $\omega(n) =r$. Let $B_1$ be the Mertens constant given by
\begin{equation*}
B_1 = \gamma + \sum_p \left( \log \left( 1 - \frac{1}{p} \right) + \frac{1}{p} \right),
\end{equation*}
with $\gamma \approx 0.57722$, the Euler-Mascheroni constant, and where the sum runs over all primes $p$. Let
\begin{equation*}
B_2 = B_1 + \sum_p \frac{1}{p(p-1)}.
\end{equation*}
The following average value formulas are well-known (see \cite[Theorem 430]{hw} and \cite[Section 1.4.4]{srf}):
\begin{equation}\label{ave_omega}
\sum_{n \leq x} \omega(n) = x \log \log x + B_1 x + O \left( \frac{x}{\log x}\right),
\end{equation}
and
\begin{equation}\label{ave_Omega}
\sum_{n \leq x} \Omega(n) = x \log \log x + B_2 x + O \left( \frac{x}{\log x}\right).
\end{equation}
Let $h \geq 2$ be an integer. Let $n$ be a natural number with the factorization given in \eqref{factorization}. We say $n$ is $h$\textit{-free} if $s_i \leq h-1$ for all $i \in \{1, \cdots, r\}$, and we say $n$ is $h$\textit{-full} if $s_i \geq h$ for all $i \in \{1, \cdots, r\}$. Let $\mathcal{S}_h$ denote the set of $h$-free numbers and $\mathcal{N}_h$ denote the set of $h$-full numbers. Let $\gamma_{0,h}$ be the constant defined as
\begin{equation}\label{gamma0h}
\gamma_{0,h} := \prod_{p} \left( 1 + \frac{p-p^{1/h}}{p^2(p^{1/h} - 1)} \right),
\end{equation}
where the product runs over all primes $p$, and $B_3$ be the constant defined as
\begin{equation}\label{B3}
    B_3 := h (B_2 - \log h) + \sum_p \frac{(h+1)p^{1+1/h} - hp - 2 h p^{2/h} + (2h-1) p^{1/h}}{(p-1) (p^{1/h} - 1) (p^{1+1/h} + p^{1/h} - p)}.
\end{equation}
In \cite[Theorem 1 and Theorem 2]{jala}, Jakimczuk and Lal\'in established the first moments of $\Omega(n)$ over $h$-free and over $h$-full numbers respectively as
\begin{equation}\label{hfreeOmega}
\sum_{\substack{n \leq x \\ n \in \mathcal{S}_h}} \Omega(n) = \frac{1}{\zeta(h)} x \log \log x + O_h(x),
\end{equation}
and
\begin{align}\label{hfullOmega}
\sum_{\substack{n \leq x \\ n \in \mathcal{N}_h}} \Omega(n) = h \gamma_{0,h} x^{1/h} \log \log x + B_3 \gamma_{0,h} x^{1/h}  + O_h \left( \frac{x^{1/h}}{\sqrt{\log x}}\right),
\end{align}
where $\zeta(s)$ represents the classical Riemann $\zeta$-function, and where $O_X$ denotes that the implied big-O constant depends on the variable set $X$. To obtain \eqref{hfreeOmega}, they used the generating Dirichlet series for $h$-free numbers. For \eqref{hfullOmega}, they used the completely additive property of $\Omega(n)$ to split the sum into sums over divisors of $n$ distinguished by their multiplicity, and then used generating series arguments to estimate these sums. 
This method for the $h$-full case cannot be extended to the study of $\omega(n)$, since $\omega(n)$ is not completely additive. This limitation of their approach is presented in \cite[Page 33, Line 15]{lz} where the authors established the function field analog of \cite{jala}. Thus, to study the moments of $\omega(n)$ over $h$-full numbers, we use a new counting argument. At its core, our method relies on counting natural numbers divisible by a given set of prime powers in a bounded range. This new counting argument not only establishes the average distribution for $\omega(n)$ over $h$-free and $h$-full numbers but also improves results from \cite{jala}. In particular, it makes the coefficient of $x$ in \eqref{hfreeOmega} explicit and improves the error term in \eqref{hfullOmega} to $x^{1/h}/\log x$. Additionally, we employ this argument to establish the second moment estimate for $\omega(n)$ over the sets of $h$-free and $h$-full numbers.

%

The first part of our work employs the counting argument to study the moments of $\omega(n)$ over the subsets of $h$-free and $h$-full numbers. As the evidence in \eqref{ave_omega} and \eqref{ave_Omega} suggests, the average distribution of $\omega(n)$ differs from that of $\Omega(n)$ in the second main term. Thus, asymptotically, they behave the same. We verify this phenomenon when restricted to these subsets.

We define the constants
\begin{equation}\label{C1}
    C_1 :=  B_1 - \sum_{p} \frac{p-1}{p(p^h -1)},
\end{equation}
and
\begin{equation*}
    C_2 := C_1^2 + C_1 - \zeta(2) - \sum_p  \left( \frac{p^{h-1} - 1}{p^h - 1} \right)^2.
\end{equation*}
Using the above definitions, we prove the first and the second moments of $\omega(n)$ over $h$-free numbers as the following:
\begin{thm}\label{hfreeomega}
Let $x > 2$ be a real number. Let $h \geq 2$ be an integer. Let $\mathcal{S}_h(x)$ be the set of $h$-free numbers less than or equal to $x$. Then, we have
$$\sum_{n \in \mathcal{S}_h(x)} \omega(n) = \frac{1}{\zeta(h)} x \log \log x + 
 \frac{C_1}{\zeta(h)} x
+ O_h \left( \frac{x}{\log x}\right),$$
and
\begin{align*}
\sum_{n \in \mathcal{S}_h(x)}  \omega^2(n) 
& = \frac{1}{\zeta(h)}  x (\log \log x)^2 + \frac{2 C_1 + 1}{\zeta(h)} x \log \log x + \frac{C_2}{\zeta(h)} x + O_h \left( \frac{x}{\log x}\right).
\end{align*}
\end{thm} 
Let $\mathcal{L}_h(r)$ be the convergent sum defined for $r > h$ as
\begin{equation}\label{lhr}
     \mathcal{L}_h(r) := \sum_p \frac{1}{p^{(r/h)-1} \left( p - p^{1-1/h} + 1 \right)},
\end{equation}
where the sum runs over all primes. Using this, we define two new constants
\begin{equation}\label{D1}
    D_1 :=  B_1 - \log h + \mathcal{L}_h(h+1) - \mathcal{L}_h(2h),
\end{equation}
and
\begin{equation}\label{D2}
    D_2 := D_1^2 + D_1 - \zeta(2)
     - \sum_{p} \left( \frac{1}{p-p^{1-1/h}+1} \right)^2.
\end{equation}
For the distributions of $\omega(n)$ over $h$-full numbers, we prove:
\begin{thm}\label{hfullomega}
Let $x > 2$ be a real number. Let $h \geq 2$ be an integer. Let $\mathcal{N}_h(x)$ be the set of $h$-full numbers less than or equal to $x$. We have
\begin{align*}
\sum_{n \in \mathcal{N}_h(x)} \omega(n) & =  \gamma_{0,h} x^{1/h} \log \log x + D_1 \gamma_{0,h} x^{1/h} + O_h \left( \frac{x^{1/h}}{\log x} \right),
\end{align*}
and
\begin{align*}
     \sum_{n \in \mathcal{N}_h(x)} \omega^2(n)
     & = \gamma_{0,h} x^{1/h} (\log \log x)^2 +   \left( 2 D_1 + 1 \right) \gamma_{0,h} x^{1/h} \log \log x + D_2 \gamma_{0,h} x^{1/h} \\
     & \hspace{.5cm} + O_h \left( \frac{x^{1/h} \log \log x}{\log x} \right).
\end{align*}
\end{thm}
Next, we introduce the definition of the normal order of an arithmetic function over a subset of natural numbers. This definition is inspired by the work of Elma and Liu \cite{el} over naturals where they extended the definition of normal order using an increasing function (see \cite{hardyram}) to a definition involving a non-decreasing function. Let $S \subseteq \mathbb{N}$ and $S(x)$ denote the set of natural numbers belonging to $S$ and less than or equal to $x$. Let $|S(x)|$ denote the cardinality of $S(x)$. Let $f, F : S \rightarrow \mathbb{R}_{\geq 0}$ be two functions such that $F$ is non-decreasing. Then, $f(n)$ is said to have normal order $F(n)$ over $S$ if for any $\epsilon > 0$, the number of $n \in S(x)$ that do not satisfy the inequality
$$(1-\epsilon) F(n) \leq f(n) \leq (1+ \epsilon) F(n)$$
is $o(|S(x)|)$ as $x \rightarrow \infty$. 

The distribution in \eqref{ave_omega} proves that on average, $\omega(n)$ behaves like $\log \log n$ over rationals. In \cite{hardyram}, Hardy and Ramanujan strengthened this result further by proving that $\omega(n)$ has the normal order $\log \log n$ over natural numbers. In \cite[Section 22.11]{hw}, one can find another proof of this result using the variance of $\omega(n)$ (see \cite[(22.11.7)]{hw}).

Note that $\omega(n)$ may exhibit different normal orders over different subsets of $\mathbb{N}$. For instance, $\omega(p) = 1$ for any prime $p$, and thus $\omega(n)$ has normal order 1 over the set of all primes. However, in this work, we establish that $\omega(n)$ still has the normal order $\log \log n$ over $h$-free and over $h$-full numbers. Since the set of $h$-free numbers has a positive density in $\mathbb{N}$, the proof of normal order of $\omega(n)$ over $h$-free numbers follows from the classical case. In particular, one can establish that for any $\epsilon > 0$, the number of $n \in \mathcal{S}_h(x)$ that do not satisfy the inequality
$$(1-\epsilon) \log \log n \leq \omega(n) \leq (1+ \epsilon) \log \log n$$
is $o(|\mathcal{S}_h(x)|)$ as $x \rightarrow \infty$. On the other hand, the set of $h$-full numbers has density zero in $\mathbb{N}$ and thus does not directly follow from the classical result. However, writing an $h$-full number $n$ as $n = r_0^h r_1$ where $r_0$ is the product of all distinct prime factors of $n$ with multiplicity $h$, one can use the classical result on $\omega(r_0)$ to establish the normal order of $\omega(n)$ over $h$-full numbers. In this work, we provide another proof of this result using the variance of $\omega(n)$ over $h$-full numbers. We prove:
\begin{thm}\label{normal-order-hfull}
    Let $h \geq 2$ be an integer. Let $D_2$ be defined by \eqref{D2} and $\gamma_{0,h}$ be defined by \eqref{gamma0h}. Let $\mathcal{N}_h(x)$ be the set of $h$-full numbers less than or equal to $x$. Then, the variance of $\omega(n)$ over $h$-full numbers is given by
    $$\sum_{n \in \mathcal{N}_h(x)} (\omega(n) - \log \log n) ^2 = \gamma_{0,h} x^{1/h} \log \log x + D_2 \gamma_{0.h} x^{1/h} + O_h \left( \frac{x^{1/h} \log \log x}{\log x} \right).$$
    Let $g(x)$ be an increasing function such that $g(x) \rightarrow \infty$ as $x \rightarrow \infty$. Then the set of natural numbers $n \in \mathcal{N}_h(x)$ such that
    $$\frac{|\omega(n) - \log \log n|}{\sqrt{\log \log n}} \geq g(x),$$
    is $o(|\mathcal{N}_h(x)|)$. As a consequence, $\omega(n)$ has normal order $\log \log n$ over $h$-full numbers.
\end{thm}

The counting argument used for $\omega(n)$ can be employed in the study of the distribution of $\Omega(n)$ and produce improvement to \eqref{hfreeOmega} and \eqref{hfullOmega} as mentioned earlier. Let $C_3$ and $C_4$ be two new constants defined as
$$C_3 := B_2 - \sum_p \frac{h}{p^h-1},$$
and
$$C_4 := C_3^2 + C_3 - \zeta(2) - \sum_p  \left( \frac{p^{h} -hp + h - 1}{(p-1)(p^h - 1)} \right)^2.$$
One can improve \eqref{hfreeOmega} as the following:
\begin{thm}
Let $x > 2$ be a real number. Let $h \geq 2$ be an integer. Let $\mathcal{S}_h(x)$ be the set of $h$-free numbers less than or equal to $x$. Then, we have
\begin{equation*}
\sum_{\substack{n \in \mathcal{S}_h(x)}} \Omega(n) = \frac{1}{\zeta(h)} x \log \log x +  \frac{C_3}{\zeta(h)} x
+ O_h \left( \frac{x}{\log x}\right),
\end{equation*}
and
\begin{align*}
\sum_{n \in \mathcal{S}_h(x)}  \Omega^2(n) 
& = \frac{1}{\zeta(h)}  x (\log \log x)^2 + \frac{2 C_3 + 1}{\zeta(h)} x \log \log x + \frac{C_4}{\zeta(h)} x + O_h \left( \frac{x}{\log x}\right).
\end{align*}
\end{thm}
Recall $B_3$ from \eqref{B3}. Let $B_4$ be a new constant defined as
$$B_4 := B_3^2+ B_3 - h^2 \zeta(2) - \sum_p \left( \frac{h(p^{1/h}-1)+1}{(p^{1/h}-1) (p - p^{1-1/h} +1)} \right)^2.$$
Adopting this definition, one can improve \eqref{hfullOmega} as the following:
\begin{thm}
Let $x > 2$ be a real number. Let $h \geq 2$ be an integer. Let $\mathcal{N}_h(x)$ be the set of $h$-full numbers less than or equal to $x$. Then, we have
\begin{align*}
\sum_{n \in \mathcal{N}_h(x)} \Omega(n) & =   h \gamma_{0,h} x^{1/h} \log \log x + B_3 \gamma_{0,h} x^{1/h} + O_h \left( \frac{x^{1/h}}{\log x} \right),
\end{align*}
and
\begin{align*}
\sum_{n \in \mathcal{N}_h(x)}  \Omega^2(n) 
& =  h^2 \gamma_{0,h} x^{1/h} (\log \log x)^2 + (2 B_3 + 1) h \gamma_{0,h} x^{1/h} \log \log x + B_4 \gamma_{0,h} x^{1/h} \\
& \hspace{.5cm} + O_h \left( \frac{x^{1/h} \log \log x}{\log x} \right).
\end{align*}
\end{thm}

Using the last two theorems and the method of variance demonstrated for $\omega(n)$ in this paper, one can establish the normal order of $\Omega(n)$ over $h$-free and $h$-full numbers as the following:
\begin{cor}
    $\Omega(n)$ has normal order $\log \log n$ over $h$-free numbers and normal order $h \log \log n$ over $h$-full numbers.
\end{cor}
The proofs for the $\Omega(n)$ results are not included in this paper. Interested readers can find the complete proofs in the first author's Ph.D. thesis \cite{das}.
\section{Lemmata}
In this section, we list all the lemmas required to study the first and the second moment of $\omega(n)$ over $h$-free and over $h$-full numbers. First, we recall the following results regarding sums over primes necessary for the study:

\begin{lma}\label{primepower}\cite[Lemma 1.2]{aler}
If $\tau > 1$ be any real number. Then
$$\sum_{p \geq x} \frac{1}{p^\tau} = \frac{1}{(\tau-1)x^{\tau-1} (\log x)} + O \left( \frac{1}{x^{\tau-1} (\log x)^\tau} \right).$$
\end{lma}
\begin{lma}\label{primepowerleqx}
Let $\alpha$ be any real number satisfying $0 < \alpha < 1$. Then
$$\sum_{p \leq x} \frac{1}{p^\alpha} = O_\alpha \left( \frac{x^{1-\alpha}}{\log x} \right).$$
\end{lma}
\begin{proof}
Recall that the classical prime number theorem yields
\begin{equation}\label{PNT}
    \pi(x) := \sum_{p \leq x} 1 = \frac{x}{\log x} + O \left( \frac{x}{(\log x)^2}\right).
\end{equation}
Thus, using partial summation and integration by parts, we obtain
\begin{align*}
    \sum_{p \leq x} \frac{1}{p^\alpha} & = \frac{\pi(x)}{x^\alpha} + \alpha \int_{2^-}^{x} \frac{\pi(t)}{t^{\alpha+1}} dt \ll_{\alpha} \frac{x^{1-\alpha}}{\log x},
\end{align*}
which completes the proof.
\end{proof}
\begin{lma}\label{sumplogp}\cite[Exercise 9.4.4]{murty}
    For $x > 2$, we have
    $$\sum_{p \leq x/2} \frac{1}{p \log (x/p)} = O \left( \frac{\log \log x}{\log x} \right).$$
\end{lma}
\begin{lma}\label{saidak}
  Let $p$ and $q$ denote prime numbers. For $x >2$, we have
\begin{equation*}
    \sum_{\substack{p,q \\ pq \leq x}} \frac{1}{pq} = (\log \log x)^2 + 2 B_1 \log \log x + B_1^2 - \zeta(2) + O \left( \frac{\log \log x}{\log x} \right).
\end{equation*}
\end{lma}
\begin{proof}
    The proof follows the exact arguments of Saidak \cite[Lemma 3]{saidak} except for a small correction. We notice that $\int_{0}^1 \log(1-t)/t \ dt = -\zeta(2)$ which proves the value of $- \zeta(2)$ in the main result as opposed to $+ \zeta(2)$ stated by Saidak.
\end{proof}
Finally, we recall the following results regarding the density of a particular sequence of $h$-free numbers:
\begin{lma}\label{restrict}\cite[Lemma 3]{jala2}
Let $x > 2$ be a real number. Let $h \geq 2$ be an integer. Let $\mathcal{S}_h(x)$ be the set of $h$-free numbers less than or equal to $x$. Let $q_1,\cdots,q_r$ be prime numbers. Then, we have
$$\sum_{\substack{n \in \mathcal{S}_h(x) \\ (q_i,n) = 1, \ \forall i \in \{1, \cdots, r \} }} 1 =  \prod_{i=1}^r \left( \frac{q_i^{h} - q_i^{h-1}}{q_i^h - 1} \right) \frac{x}{\zeta(h)} + O_h \left( 2^r x^{1/h} \right).$$
\end{lma} 
\section{The first and the second moment of \texorpdfstring{$\omega(n)$}{} over \texorpdfstring{$h$}{}-free numbers}
In this section, we establish the asymptotic distribution of $\omega(n)$ over $h$-free numbers.
\begin{proof}[\textbf{Proof of \thmref{hfreeomega}}]
Let $p^k||n$ denote the property that $p^k$ divides $n$ but $p^{k+1}$ does not divide $n$. Then writing $n = p^k y$ with $(y,p)=1$ for such $n$ and using $r = \min \left\{ h-1, \left\lfloor \log x / \log p \right\rfloor \right\}$, we obtain
\begin{equation*}\label{anomega1}
    \sum_{n \in \mathcal{S}_h(x)} \omega(n) = \sum_{n \in \mathcal{S}_h(x)} \sum_{\substack{p \\ p|n}} 1 = \sum_{p \leq x} \sum_{k=1}^r \sum_{\substack{n \in \mathcal{S}_h(x) \\ p^k || n}} 1 = \sum_{p \leq x} \sum_{k=1}^r \sum_{\substack{y \in \mathcal{S}_h(x/p^k) \\ (y,p) =1 }} 1.
\end{equation*}
Now, first using \lmaref{restrict} for a single prime $p$ to the above and then using \lmaref{primepowerleqx}, we obtain
\begin{align}\label{anomega2}
    \sum_{n \in \mathcal{S}_h(x)} \omega(n) 
 & = \sum_{p \leq x} \sum_{k=1}^r \left( \frac{p^{h} - p^{h-1}}{p^k(p^h - 1)} \right) \frac{x}{\zeta(h)} + O_h \left(  \frac{x}{\log x} \right).
\end{align}
Using the bound $\lfloor x \rfloor \geq x - 1$ and \eqref{PNT}, we can write
\begin{align}\label{caser=h-1}
    \sum_{p \leq x} \sum_{k=1}^r \left( \frac{p^{h} - p^{h-1}}{p^k(p^h - 1)} \right) 
    & = \sum_{p \leq x} \sum_{k=1}^{h-1} \left( \frac{p^{h} - p^{h-1}}{p^k(p^h - 1)} \right) + O \left( \frac{1}{\log x} \right).
\end{align}
Now, using 
$$\sum_{k=1}^{h-1} \frac{p^{h} - p^{h-1}}{p^k(p^h - 1)}  =  \frac{1}{p} - \frac{p - 1}{p(p^h - 1)},$$ 
and Merten's second theorem given as
\begin{equation}\label{sum1/p}
    \sum_{p \leq x} \frac{1}{p} = \log \log x + B_1 + O \left( \frac{1}{\log x} \right),
\end{equation}
and using \lmaref{primepower} with $\tau =h$, we obtain  
\begin{align*}
    \sum_{p \leq x} \sum_{k=1}^{h-1} \left( \frac{p^{h} - p^{h-1}}{p^k(p^h - 1)} \right)
    & = \log \log x + C_1 + O_h \left( \frac{1}{\log x} \right),
\end{align*}
where $C_1$ is defined in \eqref{C1}. Combining the above with \eqref{anomega2} and \eqref{caser=h-1} completes the first part of the proof.
Next, let $r_p = \min \left\{ h-1, \left\lfloor \log x /\log p \right\rfloor \right\}$ and $r_q = \min \left\{ h-1, \left\lfloor \log x / \log q \right\rfloor \right\}$. Then, we have
\begin{equation}\label{mainpart}
    \sum_{n \in \mathcal{S}_h(x)} \omega^2(n) = \sum_{n \in \mathcal{S}_h(x)} \left( \sum_{\substack{p \\ p | n}} 1 \right)^2 = \sum_{n \in \mathcal{S}_h(x)} \omega(n) + \sum_{n \in \mathcal{S}_h(x)} \sum_{\substack{p,q \\ p^k || n, \ q^l || n, \ p \neq q }} \left( \sum_{k=1}^{r_p} \sum_{l=1}^{r_q} 1 \right),
\end{equation}
where $p$ and $q$ above denote primes. The first sum on the right-hand side is the first moment studied above. For the second sum, we rewrite the sum and use \lmaref{restrict} for two primes $p$ and $q$ to obtain
\begin{align}\label{mainpart2}
    & \sum_{n \in \mathcal{S}_h(x)} \sum_{\substack{p,q \\ p^k || n, \ q^l || n, \ p \neq q }} \left( \sum_{k=1}^{r_p} \sum_{l=1}^{r_q} 1 \right) \notag \\
    & = \sum_{\substack{p,q \\ p \neq q, \ pq \leq x}} \sum_{k=1}^{r_p} \sum_{l=1}^{r_q} \left( \left( \frac{p^{h} - p^{h-1}}{p^{k}(p^h - 1)} \right) \left( \frac{q^{h} - q^{h-1}}{q^{l}(q^h - 1)} \right) \frac{x}{\zeta(h)} + O_h \left( \frac{x^{1/h}}{p^{k/h} q^{l/h}} \right) \right). 
\end{align}
We employ \lmaref{primepowerleqx} with $\alpha = 1/h$ and \lmaref{sumplogp} 
to estimate the error term above as 
\begin{equation}\label{mainpart3}
    \sum_{\substack{p,q \\ p \neq q, \ pq \leq x}} \sum_{k=1}^{r_p} \sum_{l=1}^{r_q} \frac{x^{1/h}}{p^{k/h} q^{l/h}} \ll x^{1/h} \sum_{\substack{p,q \\ p \neq q, \ pq \leq x}} \frac{1}{(pq)^{1/h}} \ll_h \frac{x \log \log x}{\log x}.
\end{equation}
For estimating the main term in \eqref{mainpart2}, 
we consider the set $R$ defined as
\begin{equation*}
    R:= \left\{ p \leq x \ | \  \left\lfloor \frac{\log x}{\log p}  \right\rfloor < h-1 \right\}.
\end{equation*}
Using the definition of $r_p$ and $r_q$, we rewrite 
\begin{align}\label{parts}
    & \sum_{\substack{p,q \\ p \neq q, \ pq \leq x}} \sum_{k=1}^{r_p} \sum_{l=1}^{r_q} \left( \frac{p^{h} - p^{h-1}}{p^{k}(p^h - 1)} \right) \left( \frac{q^{h} - q^{h-1}}{q^{l}(q^h - 1)} \right) \notag \\
    & = \sum_{\substack{p,q \\ p \neq q, \ pq \leq x}} \sum_{k=1}^{h-1} \sum_{l=1}^{h-1} \left( \frac{p^{h} - p^{h-1}}{p^{k}(p^h - 1)} \right) \left( \frac{q^{h} - q^{h-1}}{q^{l}(q^h - 1)} \right) -  I_1  -  I_2 + I_3,
\end{align}
where
$$I_1 = \sum_{\substack{p,q \\ p \neq q, \ pq \leq x \\ p \in R}} \sum_{k=\left\lfloor \frac{\log x}{\log p}  \right\rfloor + 1}^{h-1} \sum_{l=1}^{h-1} \left( \frac{p^{h} - p^{h-1}}{p^{k}(p^h - 1)} \right) \left( \frac{q^{h} - q^{h-1}}{q^{l}(q^h - 1)} \right),$$
$$I_2 = \sum_{\substack{p,q \\ p \neq q, \ pq \leq x \\ q \in R}} \sum_{k=1}^{h-1} \sum_{l=\left\lfloor \frac{\log x}{\log q}  \right\rfloor + 1}^{h-1} \left( \frac{p^{h} - p^{h-1}}{p^{k}(p^h - 1)} \right) \left( \frac{q^{h} - q^{h-1}}{q^{l}(q^h - 1)} \right),$$
$$I_3 = \sum_{\substack{p,q \\ p \neq q, \ pq \leq x \\ p,q \in R}} \sum_{k=\left\lfloor \frac{\log x}{\log p}  \right\rfloor + 1}^{h-1} \sum_{l=\left\lfloor \frac{\log x}{\log q}  \right\rfloor + 1}^{h-1} \left( \frac{p^{h} - p^{h-1}}{p^{k}(p^h - 1)} \right) \left( \frac{q^{h} - q^{h-1}}{q^{l}(q^h - 1)} \right).$$
The sum $I_1$, $I_2$, and $I_3$ contribute to the error term. Note that $I_1 = I_2$. In fact, using \eqref{sum1/p} and $\lfloor x \rfloor \geq x-1$, we estimate
$$I_1 \ll \sum_{\substack{p,q \\ p \neq q, \ pq \leq x \\ p \in R}} \frac{1}{p^{\frac{\log x}{\log p}}} \frac{1}{q} \ll \frac{1}{x} \sum_{q \leq x} \frac{1}{q} \ll \frac{\log \log x}{x}.$$
Finally, for $I_3$, using \eqref{PNT}, we have
$$I_3 \ll \sum_{\substack{p,q \\ p \neq q, \ pq \leq x \\ p,q \in R}} \frac{1}{p^{\frac{\log x}{\log p}}} \frac{1}{q^{\frac{\log x}{\log q}}} \ll \frac{1}{x^2} \sum_{p \leq x} 1 \ll \frac{1}{x \log x}.$$
We next estimate the main term in \eqref{parts}. First, we rewrite the sum as
\begin{align}\label{part1}
    & \sum_{\substack{p,q \\ p \neq q, \ pq \leq x}} \sum_{k=1}^{h-1} \sum_{l=1}^{h-1} \left( \frac{p^{h} - p^{h-1}}{p^{k}(p^h - 1)} \right) \left( \frac{q^{h} - q^{h-1}}{q^{l}(q^h - 1)} \right) \notag \\
    & = \sum_{\substack{p,q \\ pq \leq x}} \left( \frac{p^{h-1} - 1}{p^h - 1} \right) \left( \frac{q^{h-1} - 1}{q^h - 1}\right) - \sum_{\substack{p \\ p \leq x^{1/2}}} \left( \frac{p^{h-1} - 1}{p^h - 1} \right)^2.
\end{align}
The second sum above is estimated using \lmaref{primepower} as
\begin{align}\label{part2}
    \sum_{\substack{p \\ p \leq x^{1/2}}} \left( \frac{p^{h-1} - 1}{p^h - 1} \right)^2 
    & = \sum_{\substack{p}} \left( \frac{p^{h-1} - 1}{p^h - 1} \right)^2 + O \left( \frac{1}{x^{1/2} \log x}\right).
\end{align}
For the first sum on the right-hand side in \eqref{part1}, using 
$$\frac{p^{h-1} - 1}{p^h -1} = \frac{1}{p} - \frac{p-1}{p(p^h - 1)},$$ and the symmetry in $p$ and $q$, we have
\begin{align*}
    & \sum_{\substack{p,q \\ pq \leq x}} \left( \frac{p^{h-1} - 1}{p^h - 1} \right) \left( \frac{q^{h-1} - 1}{q^h - 1}\right) \notag \\
    & = \sum_{\substack{p,q \\ pq \leq x}} \frac{1}{pq} - 2 \sum_{\substack{p,q \\ pq \leq x}} \frac{1}{p} \left( \frac{q-1}{q(q^h - 1)} \right) + \sum_{\substack{p,q \\ pq \leq x}} \left( \frac{p-1}{p(p^h - 1)} \right) \left( \frac{q-1}{q(q^h - 1)} \right).
\end{align*}
We estimate the sums on the right-hand side above separately. For the first sum, we use \lmaref{saidak}. For the second sum, we use \lmaref{primepower}, and then \eqref{sum1/p} and \eqref{PNT} to obtain
\begin{align*}
    \sum_{\substack{p,q \\ pq \leq x}} \frac{1}{p} \left( \frac{q-1}{q(q^h - 1)} \right) 
    & = \sum_{\substack{p \\ p \leq x/2}} \frac{1}{p} \left( \sum_{p} \frac{p-1}{p(p^h - 1)}  + O \left( \frac{1}{(x/p) \log (x/p)}\right)\right) \notag \\
    & = \left( \sum_{p} \frac{p-1}{p(p^h - 1)} \right) \left( \log \log x + B_1 \right) + O \left( \frac{1}{\log x} \right).
\end{align*}
For the third sum, we use \lmaref{primepower} twice, and then \lmaref{sumplogp} to obtain
\begin{align*}
     \sum_{\substack{p,q \\ pq \leq x}} \left( \frac{p-1}{p(p^h - 1)} \right) \left( \frac{q-1}{q(q^h - 1)} \right) 
    & = \left( \sum_{p} \frac{p-1}{p(p^h - 1)} \right)^2 + O \left( \frac{\log \log x}{x \log x}\right) .
\end{align*}
Combining the last three results and \lmaref{saidak}, we obtain
\begin{align*}
    & \sum_{\substack{p,q \\ pq \leq x}} \left( \frac{p^{h-1} - 1}{p^h - 1} \right) \left( \frac{q^{h-1} - 1}{q^h - 1}\right) = (\log \log x)^2 + 2 C_1 \log \log x + C_1^2 - \zeta(2) + O \left( \frac{\log \log x}{\log x} \right).
\end{align*}
Combining \eqref{mainpart}, \eqref{mainpart2}, \eqref{mainpart3}, \eqref{parts}, \eqref{part1}, and \eqref{part2} with the above equation and using the first-moment estimate, we obtain the required second moment.
\end{proof}
\section{The first and the second moment of \texorpdfstring{$\omega(n)$}{} over \texorpdfstring{$h$}{}-full numbers}
Let $h \geq 2$ be an integer. Recall that $\mathcal{N}_h$ is the set of $h$-full numbers and $\mathcal{N}_h(x)$ is the set of $h$-full numbers less than or equal to $x$. 
Note that, if $n \in \mathcal{N}_h(x)$, then for any prime $p | n$, we must have $p^h | n$. Moreover, the maximal power of $p$ that divides any such $n \leq x$ is $\lfloor \log x /\log p \rfloor$. Using this and $n = p^k y$ with $(y,p) =1$ when $p^k || n$, we can write
\begin{equation}\label{omegaonhfull}
\sum_{n \in \mathcal{N}_h(x)} \omega(n) = \sum_{n \in \mathcal{N}_h(x)} \sum_{\substack{p \\ p | n}} 1 = \sum_{p \leq x^{1/h}} \sum_{k = h}^{\left\lfloor \frac{\log x}{\log p}  \right\rfloor} \sum_{\substack{y \in \mathcal{N}_h(x/p^k) \\(y,p) =1}} 1.
\end{equation}

Let $q$ be a fixed prime. We intend to estimate the sum
\begin{equation}\label{Aqhx}
A_{q,h}(y) = \sum_{\substack{n \in \mathcal{N}_h(y) \\ (n,q) =1}} 1.
\end{equation}
To prove the estimate for the above sum, we will use techniques from 
the work of Ivi\'c and Shiu \cite{is}. Let's recall some of the results about $h$-full numbers from their work:

Notice that the generating series for the $h$-full numbers is defined on $\Re(s) > 1/h$ as:
$$N_h(s) = \sum_{n \in \mathcal{N}_h} \frac{1}{n^s} = \prod_p \left( 1 + \frac{1}{p^{hs}} + \cdots + \frac{1}{p^{ks}} + \cdots \right) = \prod_p \left( 1 +  \frac{p^{-hs}}{1 - p^{-s}} \right).$$
Let $A_h(y)$ denote the number of $h$-full integers not exceeding $y$.
In \cite[Section 1]{is}, Ivi\'c and Shiu proved that
\begin{equation}\label{Nhxformula}
N_h(s) = \zeta(hs) \zeta((h+1)s) \cdots \zeta((2h-1)s) \zeta^{-1}((2h+2)s) \phi_h(s),
\end{equation}
where $\phi_h(s)$ satisfies
\begin{equation}\label{iden1}
\prod_{p} \left( 1 - p^{-(2h+2)s} + \sum_{r = 2h+3}^{(3h^2+h-2)/2} a_{r,h} p^{-rs} \right) = \zeta^{-1}((2h+2)s) \phi_h(s),
\end{equation}
and where $a_{r,h}$ satisfy the identity
\begin{equation}\label{iden2}
    \left( 1 + \frac{v^h}{1-v} \right)(1-v^h) \cdots (1-v^{2h-1}) = 1 - v^{2h+2} + \sum_{2h+3}^{(3h^2+h-2)/2} a_{r,h} v^r.
\end{equation}
Note that $|a_{r,h}| \leq h 2^h$, $\phi_2(s) = 1$, and $\phi_h(s)$ has a Dirichlet series with abscissa of absolute convergence equal to $1/(2h+3)$ if $h > 2$. Moreover, it is also established that
$$N_h(s) = G_h(s) K_h(s),$$
where
\begin{equation}\label{K_h(s)}
K_h(s) = \sum_{n=1}^\infty k_h(n) n^{-s} = \zeta(hs) \zeta((h+1)s) \cdots \zeta((2h-1)s),
\end{equation}
and
\begin{equation*}
G_h(s) = \sum_{n=1}^\infty g_h(n) n^{-s} = \frac{\phi_h(s)}{\zeta((2h+2)s)},
\end{equation*}
is a Dirichlet series converging absolutely in $\Re(s) > 1/(2h+2)$. Thus, one can write
$$A_h(y) 
= \sum_{mn \leq y} g_h(m) k_h(n).$$
Additionally, let $S_h(y)$ denote the sum
\begin{equation}\label{S_h(x)}
S_h(y) := \sum_{n \leq y} k_h(n).
\end{equation}
In \cite[Theorem 1]{is}, Ivi\'c and Shiu proved that 
\begin{equation}\label{s_h(x)}
S_h(y) = \sum_{r=h}^{2h-1} C_{r,h} y^{1/r} + \Delta_h^*(y),
\end{equation}
where 
\begin{equation}\label{crh}
    C_{r,h} = \prod_{j=h, j \neq r}^{2h-1} \zeta(j/r),
\end{equation}
and where
\begin{equation}\label{bounddeltah*}
\Delta_h^*(y) \ll y^{\eta_h},    
\end{equation}
with $1/(2h+2) < \eta_h < 1/(2h -1)$.
Using the above pre-requisite results, we establish an asymptotic result for $A_{q,h}(y)$ as defined in \eqref{Aqhx}. We prove:
\begin{lma}\label{Aqhlemma}
Let $q_1,q_2,\cdots,q_r$ be distinct primes. Let
$$A_{q_1,\cdots,q_r,h}(x) := \sum_{\substack{n \in \mathcal{N}_h(x) \\ (n,q_1) = \cdots = (n,q_r) =1}} 1.$$
For any $x >2$, we have
\begin{equation*}
    A_{q_1,\cdots,q_r,h}(x) = \gamma_{q_1,\cdots,q_r,0,h} x^{\frac{1}{h}} + \gamma_{q_1,\cdots,q_r,1,h} x^{\frac{1}{h+1}} + \cdots + \gamma_{q_1,\cdots,q_r,h-1,h} x^{\frac{1}{2h-1}} + O_h \left( x^{\eta_h} \right),
\end{equation*}
where $1/(2h+2) < \eta_h < 1/(2h -1)$, and for $i \in \{0,1,\cdots,h-1\}$,
\begin{equation}\label{gammaqh}
\gamma_{q_1,\cdots,q_r,i,h} = C_{h+i,h} \frac{\phi_h(1/(h+i))}{\zeta((2h+2)/(h+i)) \left( \prod_{j=1}^r\left( 1 +  \frac{q_j^{-h/(h+i)}}{1 - q_j^{-1/(h+i)}} \right) \right)},
\end{equation}
with $C_{h+i,h}$ defined in \eqref{crh} and $\phi_h(s)$ satisfying \eqref{iden1}.
%
\end{lma}
\begin{proof}
Let us consider the sum defined on $\Re(s) > 1/h$ as
$$N_{q_1,\cdots,q_r,h}(s) 
= \sum_{\substack{n \in \mathcal{N}_h \\ (n,q_1) = \cdots = (n,q_r) =1}} \frac{1}{n^s} 
= \prod_{\substack{p \\ p \neq q_1,\cdots,q_r}} \left( 1 +  \frac{p^{-hs}}{1 - p^{-s}} \right).$$
By \eqref{Nhxformula}, we have
$$N_{q_1,\cdots,q_r,h}(s) = \zeta(hs) \zeta((h+1)s) \cdots \zeta((2h-1)s) \frac{ \zeta^{-1}((2h+2)s) \phi_h(s)}{\prod_{j=1}^r \left( 1 +  \frac{q_j^{-hs}}{1 - q_j^{-s}} \right)}.$$
Thus, we can write
$$N_{q_1,\cdots,q_r,h}(s) = G_{q_1,\cdots,q_r,h}(s) K_h(s),$$
where $K_h(s)$ is defined in \eqref{K_h(s)} and 
$$G_{q,h}(s) = \sum_{n=1}^\infty g_{q_1,\cdots,q_r,h}(n) n^{-s} = \frac{\phi_h(s)}{\zeta((2h+2)s) \left( \prod_{j=1}^r \left( 1 +  \frac{q_j^{-hs}}{1 - q_j^{-s}} \right)\right)}.$$
By \eqref{iden1} and \eqref{iden2}, we have
\begin{align}\label{iden3}
    & \frac{\phi_h(s)}{\zeta((2h+2)s) \left( \prod_{j=1}^r \left( 1 +  \frac{q_j^{-hs}}{1 - q_j^{-s}} \right)\right)} \notag \\ 
    & = \prod_{\substack{p \\ p \neq q_1,\cdots,q_r}} \left( 1 - p^{-(2h+2)s} + \sum_{r = 2h+3}^{(3h^2+h-2)/2} a_{r,h} p^{-rs} \right) \left( \prod_{j=1}^r \prod_{r=h}^{2h-1} \left( 1 - q_j^{-rs}\right) \right).
\end{align}
Note that, for $\Re(s) > 1/(2h+2)$ and for all sufficiently large $p$, using $|a_{r,h}| \leq h 2^h$, we have
$$\left| p^{-(2h+2)s} + \sum_{r = 2h+3}^{(3h^2+h-2)/2} a_{r,h} p^{-rs}  \right| < 1.$$
Moreover, $|q_j^{-rs}| < 1$ for all $j=1, \cdots, r$ and for all $r \in \{h, h+1, \cdots, 2h-1 \}$. Thus, taking the logarithm of the right-hand side of \eqref{iden3} in the region $\Re(s) > 1/(2h+2)$ and using $|a_{r,h}| \leq h 2^h$ again, we obtain 
\begin{align*}
    & \log \left( \prod_{\substack{p \\ p \neq q_1,\cdots,q_r}} \left( 1 - p^{-(2h+2)s} + \sum_{r = 2h+3}^{(3h^2+h-2)/2} a_{r,h} p^{-rs} \right) \left( \prod_{j=1}^r \prod_{r=h}^{2h-1} \left( 1 - q_j^{-rs}\right) \right) \right) \\
    & = \sum_{\substack{p \\ p \neq q_1,\cdots,q_r}} \log \left( 1 - p^{-(2h+2)s} + \sum_{r = 2h+3}^{(3h^2+h-2)/2} a_{r,h} p^{-rs} \right) + \sum_{j=1}^r \sum_{r = h}^{2h-1} \log \left( 1 - q_j^{-rs} \right) \\
    & \ll_{h,r} \sum_{p} \frac{1}{p^{(2h+2) \Re(s)}} + \frac{1}{2^{h \Re(s)}}.
\end{align*}
The last term of the previous equation is bounded with $\ll_h \zeta((2h+2) \Re(s)) + 1/2^{h \Re(s)}$ and thus is finite for $\Re(s) > 1/(2h+2)$. Hence, by the theory of convergence of infinite products, $G_{q,h}(s)$ converges absolutely in $\Re(s) > 1/(2h+2)$. 

Now, we can write
$$A_{q_1,\cdots,q_r,h}(x) = \sum_{mn \leq x} g_{q_1,\cdots,q_r,h}(m) k_h(n) = \sum_{m \leq x} g_{q_1,\cdots,q_r,h}(m) S_h(x/m),$$
where $S_h$ is defined in \eqref{S_h(x)}. Further, using \eqref{s_h(x)}, we obtain
\begin{equation*}
A_{q_1,\cdots,q_r,h}(x) = \sum_{r=h}^{2h-1} C_{r,h} x^{1/r} \left( \sum_{m \leq x} \frac{g_{q_1,\cdots,q_r,h}(m)}{m^{1/r}} \right) +  \sum_{m \leq x} g_{q_1,\cdots,q_r,h}(m) \Delta_h^*(x/m).
\end{equation*}
We need to estimate the sums above. 
Note that
$$\sum_{n \leq x} g_{q_1,\cdots,q_r,h}(n) n^{-\eta_h} = O_h(1)
.$$
Thus, for each $r = h, h+1, \cdots, 2h-1$, using partial summation, we obtain
\begin{equation*}
\sum_{m \leq x} \frac{g_{q_1,\cdots,q_r,h}(m)}{m^{1/r}} = G_{q_1,\cdots,q_r,h}(1/r) - \sum_{m > x} \frac{g_{q_1,\cdots,q_r,h}(m)}{m^{1/r}} = G_{q_1,\cdots,q_r,h}(1/r) + O_h \left( x^{\eta_h - \frac{1}{r}} \right),
\end{equation*}
and by \eqref{bounddeltah*},
$$\sum_{m \leq x} g_{q_1,\cdots,q_r,h}(m) \Delta_h^*(x/m) \ll x^{\eta_h} \sum_{m \leq x} |g_{q_1,\cdots,q_r,h}(m)| m^{-\eta_h} \ll_h x^{\eta_h}.$$
Combining the above results, we obtain
\begin{align*}
    A_{q_1,\cdots,q_r,h}(x) & = \sum_{r=h}^{2h-1} C_{r,h} G_{q_1,\cdots,q_r,h}(1/r) x^{1/r} + O_h \left( x^{\eta_h} \right) \\
    & = \gamma_{q_1,\cdots,q_r,0,h} x^{1/h} + \cdots + \gamma_{q_1,\cdots,q_r,h-1,h} x^{1/(2h-1)} + O_h \left( x^{\eta_h} \right),
\end{align*}
which completes the proof.
\end{proof}
\begin{rmk}
Note that
\begin{equation}\label{gamma0heq}
\gamma_{q_1,\cdots,q_r,0,h} = \frac{\gamma_{0,h}}{\prod_{j=1}^r \left( 1+ \frac{q_j^{-1}}{1-q_j^{-1/h}} \right)},
\end{equation}
where 
\begin{equation}\label{gamma0_hmain}
\gamma_{0,h} = C_{h,h} \frac{\phi_h(1/h)}{\zeta((2h+2)/h)} = \prod_{p} \left( 1 + \frac{p - p^{1/h}}{p^2(p^{1/h}-1)} \right),
\end{equation}
with the product formula described in \cite[Page 3]{jala}. 
\end{rmk}
We use \lmaref{Aqhlemma} with one prime to prove the first and the second moment estimates for $\omega(n)$ over h-full numbers. We prove:
\begin{proof}[\textbf{Proof of \thmref{hfullomega}}]
Inserting the formula for $A_{q,h}(y)$ given in \lmaref{Aqhlemma} with $y = x/p^k$ in \eqref{omegaonhfull}, we obtain
\begin{align}\label{mainomegahfull1}
& \sum_{n \in \mathcal{N}_h(x)} \omega(n) \notag\\
& = \sum_{p \leq x^{1/h}} \sum_{k = h}^{\left\lfloor \frac{\log x}{\log p}  \right\rfloor} A_{p,h}(x/p^k) \notag \\
& = \sum_{p \leq x^{1/h}} \sum_{k = h}^{\left\lfloor \frac{\log x}{\log p}  \right\rfloor} \left( \gamma_{p,0,h} (x/p^k)^{1/h} + \cdots + \gamma_{p,h-1,h} (x/p^k)^{1/(2h-1)} + O_h \left( (x/p^k)^{\eta} \right) \right).
\end{align}
Let us study the main term above given as
$$\sum_{p \leq x^{1/h}} \sum_{k = h}^{\left\lfloor \frac{\log x}{\log p}  \right\rfloor} \gamma_{p,0,h} (x/p^k)^{1/h} = \gamma_{0,h} x^{1/h} \sum_{p \leq x^{1/h}} \sum_{k = h}^{\left\lfloor \frac{\log x}{\log p}  \right\rfloor} \left( \frac{1}{p^{k/h} \left( 1 + \frac{p^{-1}}{1-p^{-1/h}}\right)} \right).$$
We obtain
\begin{align}\label{sum_need}
& \gamma_{0,h} x^{1/h} \sum_{p \leq x^{1/h}} \sum_{k = h}^{\left\lfloor \frac{\log x}{\log p}  \right\rfloor} \left( \frac{1}{p^{k/h} \left( 1 + \frac{p^{-1}}{1-p^{-1/h}}\right)} \right) \notag \\
& = \gamma_{0,h} x^{1/h} \sum_{p \leq x^{1/h}} \frac{1}{p\left( 1- p^{-1/h}+p^{-1} \right)} - \gamma_{0,h} x^{1/h} \sum_{p \leq x^{1/h}} \frac{(p^{-1/h})^{\left\lfloor \frac{\log x}{\log p}  \right\rfloor - h +1}}{p\left( 1- p^{-1/h}+p^{-1} \right)}.
\end{align}
Using $\lfloor x \rfloor \geq x-1$, and then \eqref{PNT} with $y = x^{1/h}$, we bound the second term above with
\begin{equation}\label{needme}
\gamma_{0,h} x^{1/h} \sum_{p \leq x^{1/h}} \frac{(p^{-1/h})^{\left\lfloor \frac{\log x}{\log p}  \right\rfloor - h +1}}{p\left( 1- p^{-1/h}+p^{-1} \right)} \ll_h \frac{x^{1/h}}{\log x}.
\end{equation}
Thus, it remains to study the first term in \eqref{sum_need}. A little manipulation yields
\begin{align}\label{req3}
\sum_{p \leq x^{1/h}} \frac{1}{p\left( 1- p^{-1/h}+p^{-1} \right)}  
& = \sum_{p \leq x^{1/h}} \frac{1}{p} + \sum_{p \leq x^{1/h}}  \frac{p^{-1/h} - p^{-1}}{ p \left( 1- p^{-1/h}+p^{-1} \right) }.
\end{align}
For the first sum on the right-hand side above, we use \eqref{sum1/p} to obtain
\begin{equation}\label{req2}
\sum_{p \leq x^{1/h}} \frac{1}{p} = \log \log x - \log h + B_1 + O_h \left( \frac{1}{\log x} \right).
\end{equation}
For the second sum on the right-hand side of \eqref{req3}, we use the convergence of the corresponding infinite series and \lmaref{primepower} to obtain
\begin{align*}
\sum_{p \leq x^{1/h}}  \frac{p^{-1/h} - p^{-1}}{ p \left( 1- p^{-1/h}+p^{-1} \right) } 
& = \mathcal{L}_h(h+1) - \mathcal{L}_h(2h) + O_h \left( \frac{1}{x^{1/h^2} (\log x)} \right).
\end{align*}  

Combining the last three results, we obtain
\begin{align}\label{req_bound_p}
    \sum_{p \leq x^{1/h}} \frac{1}{p\left( 1- p^{-1/h}+p^{-1} \right)} & = \log \log x + D_1
    + O_h \left( \frac{1}{\log x} \right),
\end{align}
where $D_1$ is defined in \eqref{D1}. Combining the above result with \eqref{mainomegahfull1}, \eqref{sum_need}, and \eqref{needme}, we obtain
\begin{align}\label{hfullfinal}
\sum_{n \in \mathcal{N}_h(x)} \omega(n) & =  \gamma_{0,h} x^{1/h} \log \log x + D_1 \gamma_{0,h} x^{1/h}  + O_h \left( \frac{x^{1/h}}{\log x} \right)  \notag \\
& \hspace{.5cm} + O_h \left( \sum_{p \leq x^{1/h}} \sum_{k = h}^{\left\lfloor \frac{\log x}{\log p}  \right\rfloor} (x/p^k)^{1/(h+1)}\right).
\end{align}
For the error term above, using \lmaref{primepowerleqx} with $\alpha = h/(h+1)$, we obtain
$$\sum_{p \leq x^{1/h}} \sum_{k = h}^{\left\lfloor \frac{\log x}{\log p}  \right\rfloor} (x/p^k)^{1/(h+1)} \ll x^{1/(h+1)} \sum_{p \leq x^{1/h}} \frac{1}{p^{h/(h+1)}} \ll_h \frac{x^{1/h}}{\log x}.$$

Inserting the above back into \eqref{hfullfinal} completes the first part of the proof.

Now, note that
\begin{equation}\label{hfullp1}
    \sum_{n \in \mathcal{N}_h(x)} \omega^2(n)
    = \sum_{n \in \mathcal{N}_h(x)} \left( \sum_{\substack{p \\ p^h | n}} 1 \right)^2 
    = \sum_{n \in \mathcal{N}_h(x)} \omega(n) + \sum_{n \in \mathcal{N}_h(x)}  \sum_{\substack{p,q \\ p^k || n, \ q^l || n, \ p \neq q }} \left( \sum_{k=h}^{\left\lfloor \frac{\log x}{\log p}  \right\rfloor} \sum_{l=h}^{\left\lfloor \frac{\log x}{\log q}  \right\rfloor} 1 \right).
\end{equation}
The first sum on the right side above is the first moment studied earlier. For the second sum, we first rewrite the sum, use \lmaref{Aqhlemma} with two distinct primes $p$ and $q$, and use \eqref{gamma0heq} to obtain
\begin{align}\label{hfullp2}
    & \sum_{n \in \mathcal{N}_h(x)}  \sum_{\substack{p,q \\ p^k || n, \ q^l || n, \ p \neq q }} \left( \sum_{k=h}^{\left\lfloor \frac{\log x}{\log p}  \right\rfloor} \sum_{l=h}^{\left\lfloor \frac{\log x}{\log q}  \right\rfloor} 1 \right) \notag \\
    & = \gamma_{0,h} x^{1/h} \sum_{\substack{p,q \\ p \neq q, \ pq \leq x^{1/h}}} \sum_{k=h}^{\left\lfloor \frac{\log x}{\log p} \right\rfloor} \sum_{l=h}^{\left\lfloor \frac{\log x}{\log q} \right\rfloor} \frac{1}{p^{k/h} \left( 1 + \frac{p^{-1}}{1 - p^{-1/h}} \right)} \frac{1}{q^{l/h}  \left( 1 + \frac{q^{-1}}{1 - q^{-1/h}} \right)} \notag \\
    & \hspace{.5cm} + O_h \left( x^\frac{1}{h+1}   \sum_{\substack{p,q \\ p \neq q, \ pq \leq x^{1/h}}} \sum_{k=h}^{\left\lfloor \frac{\log x}{\log p} \right\rfloor} \sum_{l=h}^{\left\lfloor \frac{\log x}{\log q} \right\rfloor} \frac{1}{p^{k/(h+1)} q^{l/(h+1)}}\right).
\end{align}
The error term above is bounded by using \lmaref{primepowerleqx} and then \lmaref{sumplogp} as the following
\begin{align}\label{hfullp3}
     x^\frac{1}{h+1} \sum_{\substack{p,q \\ p \neq q, \ pq \leq x^{1/h}}} \sum_{k=h}^{\left\lfloor \frac{\log x}{\log p} \right\rfloor} \sum_{l=h}^{\left\lfloor \frac{\log x}{\log q} \right\rfloor}  \frac{1}{p^{k/(h+1)} q^{l/(h+1)}} 
     & \ll_h x^{\frac{1}{h+1}} \sum_{\substack{p \\ p \leq x^{1/h}/2}} \frac{1}{p^{h/(h+1)}}  \sum_{\substack{q \\ q \leq x^{1/h}/p}} \frac{1}{q^{h/(h+1)}}  \notag \\
     & \ll_h x^{\frac{1}{h}} \sum_{\substack{p \\ p \leq x^{1/h}/2}} \frac{1}{p \log(x^{1/h}/p)} \notag\\
     & \ll_h \frac{x^{\frac{1}{h}} \log \log x}{\log x}.
\end{align}
Next, we estimate the main term in \eqref{hfullp2}. First note that
$$\sum_{k=h}^{\left\lfloor \frac{\log x}{\log p} \right\rfloor} \frac{1}{p^{k/h} \left( 1 + \frac{p^{-1}}{1 - p^{-1/h}} \right)} = \frac{1}{p(1-p^{-1/h}+p^{-1})} - \frac{p^{-\frac{1}{h} \left( \left\lfloor \frac{\log x}{\log p} \right\rfloor - h + 1\right)}}{p(1-p^{-1/h}+p^{-1})}.$$
Thus, using a similar result for a prime $q$ as above and the symmetry of two primes $p$ and $q$, we deduce
\begin{align*}
    & \gamma_{0,h} x^{1/h} \sum_{\substack{p,q \\ p \neq q, \ pq \leq x^{1/h}}} \sum_{k=h}^{\left\lfloor \frac{\log x}{\log p} \right\rfloor} \sum_{l=h}^{\left\lfloor \frac{\log x}{\log q} \right\rfloor} \frac{1}{p^{k/h} \left( 1 + \frac{p^{-1}}{1 - p^{-1/h}} \right)} \frac{1}{q^{l/h}  \left( 1 + \frac{q^{-1}}{1 - q^{-1/h}} \right)} \\
    & = \gamma_{0,h} x^{1/h} \sum_{\substack{p,q \\ p \neq q, \ pq \leq x^{1/h}}} \frac{1}{p(1-p^{-1/h}+p^{-1})} \frac{1}{q(1-q^{-1/h}+q^{-1})} - 2 J_1 + J_2,
\end{align*}
where
$$J_1 = \gamma_{0,h} x^{1/h} \sum_{\substack{p,q \\ p \neq q, \ pq \leq x^{1/h}}} \frac{1}{p(1-p^{-1/h}+p^{-1})}   \frac{q^{-\frac{1}{h} \left( \left\lfloor \frac{\log x}{\log q} \right\rfloor - h + 1\right)}}{q(1-q^{-1/h}+q^{-1})},$$
and
$$J_2 = \gamma_{0,h} x^{1/h} \sum_{\substack{p,q \\ p \neq q, \ pq \leq x^{1/h}}} \frac{p^{-\frac{1}{h} \left(  \left\lfloor \frac{\log x}{\log p} \right\rfloor - h + 1\right)}}{p(1-p^{-1/h}+p^{-1})} \frac{q^{-\frac{1}{h} \left( \left\lfloor \frac{\log x}{\log q} \right\rfloor  - h + 1\right)}}{q(1-q^{-1/h}+q^{-1})}.$$
Using $\lfloor x \rfloor \geq x-1$, \eqref{PNT} and \lmaref{sumplogp}, we obtain
\begin{align*}
    J_1 & \ll_h  \sum_{\substack{p \\ p \leq x^{1/h}/2}} \frac{1}{p(1-p^{-1/h}+p^{-1})} \left( \sum_{\substack{q \\ q \leq x^{1/h}/p}} 1 \right)
    \ll_h \frac{x^{1/h} \log \log x}{\log x}.
\end{align*}
Similarly, we obtain
\begin{align*}
    J_2 & \ll_h \sum_{\substack{p \\ p \leq x^{1/h}/2}} \frac{p^{-\frac{1}{h} \left( \left\lfloor \frac{\log x}{\log p} \right\rfloor - h + 1\right)}}{p (1-p^{-1/h}+p^{-1})} \left( \sum_{\substack{q \\ q \leq x^{1/h}/p}} 1 \right) 
    \ll_h \frac{\log \log x}{\log x}.
\end{align*}
Combining the last three results, we obtain
\begin{align}\label{hfullp4}
    & \gamma_{0,h} x^{1/h} \sum_{\substack{p,q \\ p \neq q, \ pq \leq x^{1/h}}} \sum_{k=h}^{\left\lfloor \frac{\log x}{\log p} \right\rfloor} \sum_{l=h}^{\left\lfloor \frac{\log x}{\log q} \right\rfloor} \frac{1}{p^{k/h} \left( 1 + \frac{p^{-1}}{1 - p^{-1/h}} \right)} \frac{1}{q^{l/h}  \left( 1 + \frac{q^{-1}}{1 - q^{-1/h}} \right)} \notag \\
    & = \gamma_{0,h} x^{1/h} \sum_{\substack{p,q \\ p \neq q, \ pq \leq x^{1/h}}} \frac{1}{p(1-p^{-1/h}+p^{-1})} \frac{1}{q(1-q^{-1/h}+q^{-1})} + O_h \left( \frac{x^{1/h} \log \log x}{\log x} \right).
\end{align}
Thus to complete the proof, we only require to estimate the main term in \eqref{hfullp4}. Using \lmaref{primepower}, we have
\begin{align}\label{hfullp5}
    & \sum_{\substack{p, q \\ p \neq q, \ pq \leq x^{1/h}}} \frac{1}{p(1-p^{-1/h}+p^{-1})} \frac{1}{q(1-q^{-1/h}+q^{-1})} \notag \\
    & = \sum_{\substack{p,q \\ pq \leq x^{1/h}}} \frac{1}{p(1-p^{-1/h}+p^{-1})} \frac{1}{q(1-q^{-1/h}+q^{-1})} - \sum_{p} \left( \frac{1}{p-p^{1-1/h}+1} \right)^2 \notag \\
    & \hspace{.5cm} + O_h \left( \frac{1}{x^{1/(2h)} \log x} \right).
\end{align}
Now, using
$$\frac{1}{p(1-p^{-1/h}+p^{-1})} = \frac{1}{p} + \frac{p^{-1/h} - p^{-1}}{p(1-p^{-1/h}+p^{-1})},$$
a similar result for another prime $q$, and the symmetry of primes $p$ and $q$, we write the first sum on the right-hand side of \eqref{hfullp5} as
\begin{align*}
    & \sum_{\substack{p,q \\ pq \leq x^{1/h}}} \frac{1}{p(1-p^{-1/h}+p^{-1})} \frac{1}{q(1-q^{-1/h}+q^{-1})} \\
    & = \sum_{\substack{p,q \\ pq \leq x^{1/h}}} \frac{1}{pq} + 2 \sum_{\substack{p,q \\ pq \leq x^{1/h}}} \frac{q^{-1/h} - q^{-1}}{p q(1-q^{-1/h}+q^{-1})} \\
    & \hspace{.5cm} + \sum_{\substack{p,q \\ pq \leq x^{1/h}}} \frac{p^{-1/h} - p^{-1}}{p(1-p^{-1/h}+p^{-1})} \frac{q^{-1/h} - q^{-1}}{q(1-q^{-1/h}+q^{-1})}.
\end{align*}
The first sum on the right-hand side above is estimate using \lmaref{saidak}. For the second sum, we use \lmaref{primepower}, \lmaref{sumplogp} and \eqref{sum1/p} to obtain
\begin{align*}
    & \sum_{\substack{p,q \\ pq \leq x^{1/h}}} \frac{q^{-1/h} - q^{-1}}{p q(1-q^{-1/h}+q^{-1})} \\
    & = \sum_{\substack{p \\ p \leq x^{1/h}/2}} \frac{1}{p} \sum_{\substack{q \\ q \leq x^{1/h}/p}} \frac{q^{-1/h} - q^{-1}}{q(1-q^{-1/h}+q^{-1})} \\
    & = \left( \mathcal{L}_h(h+1) - \mathcal{L}_h(2h) \right) \left(  \log \log x + B_1 - \log h \right) + O_h \left( \frac{\log \log x}{\log x} \right),
\end{align*}
and similarly, for the third sum, we obtain 
\begin{align*}
    & \sum_{\substack{p,q \\ pq \leq x^{1/h}}} \frac{p^{-1/h} - p^{-1}}{p(1-p^{-1/h}+p^{-1})} \frac{q^{-1/h} - q^{-1}}{q(1-q^{-1/h}+q^{-1})} \\
    & =  \left( \mathcal{L}_h(h+1) - \mathcal{L}_h(2h) \right)^2 + O_h \left( \frac{\log \log x}{x^{1/h^2} \log x} \right).
\end{align*}
Combining the last three results with \lmaref{saidak}, we obtain
\begin{align*}
     & \sum_{\substack{p,q \\ pq \leq x^{1/h}}} \frac{1}{p(1-p^{-1/h}+p^{-1})} \frac{1}{q(1-q^{-1/h}+q^{-1})} \notag \\
     & = (\log \log x)^2 + 2 D_1 \log \log x + D_1^2 - \zeta(2)
     + O_h \left( \frac{\log \log x}{\log x} \right).
\end{align*}
Combining the above with \eqref{hfullp1}, \eqref{hfullp2}, \eqref{hfullp3}, \eqref{hfullp4}, and \eqref{hfullp5}, and using the first moment completes the proof for the second moment.
\end{proof}

\section{Normal order of \texorpdfstring{$\omega(n)$}{} over \texorpdfstring{{h}}{}-full numbers}
In this section, using the variance of $\omega(n)$ over the set of $h$-full numbers, we prove that $\omega(n)$ has normal order $\log \log n$ when restricted to this set.

\begin{proof}[\textbf{Proof of \thmref{normal-order-hfull}}]
Notice that
\begin{equation}\label{normaleq2}
   \sum_{n \in \mathcal{N}_h(x)} (\omega(n) - \log \log n) ^2 = \sum_{n \in \mathcal{N}_h(x)} \omega^2(n) - 2 \sum_{n \in \mathcal{N}_h(x)} \omega(n) \log \log n + \sum_{n \in \mathcal{N}_h(x)} (\log \log n)^2.
\end{equation}
It is also well-known that (see \cite[Lemma 1]{is}) \begin{equation}\label{nhxbound}
    |\mathcal{N}_h(x)| = \gamma_{0,h} x^{1/h} + O_h(x^{1/(h+1)}),
\end{equation}
where $\gamma_{0,h}$ is defined in \eqref{gamma0h}. Using the above with the first moment of $\omega(n)$ over $h$-full numbers given in \thmref{hfullomega} and the partial summation formula, we obtain
$$\sum_{n \in \mathcal{N}_h(x)} \omega(n) \log \log n =  \gamma_{0,h} x^{1/h} (\log \log x)^2 + \gamma_{0,h}
 D_1 x^{1/h}  \log \log x
+ O_h \left( \frac{x^{1/h} \log \log x}{\log x}\right),$$
and
$$\sum_{n \in \mathcal{N}_h(x)} (\log \log n)^2 = \gamma_{0,h} x^{1/h} (\log \log x)^2 + O_h \left( \frac{x^{1/h} \log \log x}{\log x}\right).$$
Using the above two results and the second moment of $\omega(n)$ over $h$-full numbers given in \thmref{hfullomega} in \eqref{normaleq2}, we obtain
$$\sum_{n \in \mathcal{N}_h(x)} (\omega(n) - \log \log n) ^2 = \gamma_{0,h} x^{1/h} \log \log x + D_2 \gamma_{0.h} x^{1/h} + O_h \left( \frac{x^{1/h} \log \log x}{\log x} \right).$$

For the second part, let $g(x)$ be an increasing function such that $g(x) \rightarrow \infty$ as $x \rightarrow \infty$, and let $E_{\text{full}}^h$ be the set of natural numbers $n \in \mathcal{N}_h$ with $x/\log x < n \leq x$ such that
$$\frac{|\omega(n) - \log \log n|}{\sqrt{\log \log n}} \geq g(x).$$
Let $|E_{\text{full}}^h|$ be the cardinality of $E_{\text{full}}^h$. Then
\begin{align*}
    \sum_{n \in \mathcal{N}_h(x)} (\omega(n) - \log \log n) ^2 
    & > g^2(x/\log x) |E_{\text{full}}^h| \log \log (x/\log x).
\end{align*}
Using the asymptotic result for the left-hand side above, we deduce
$$\frac{|E_{\text{full}}^h|}{x^{1/h}} \ll_h \frac{\log \log x}{g^2(x/\log x) \log \log (x/\log x)}.$$
Since $g(x) \rightarrow \infty$ as $x \rightarrow \infty$, the right-hand side above goes to 0 as $x \rightarrow \infty$. Thus, $|E_{\text{full}}^h| = o(x^{1/h})$. Note that, by \eqref{nhxbound},  $|\mathcal{N}_h(x)| \sim \gamma_{0,h} x^{1/h}$. Thus the set of natural numbers $n \in \mathcal{N}_h$ with $x/ \log x < n \leq x$ such that
$$\frac{|\omega(n) - \log \log n|}{\sqrt{\log \log n}} \geq g(x),$$
is $o(|\mathcal{N}_h(x)|)$. Again, by \eqref{nhxbound}, the remaining set of natural numbers $n \in \mathcal{N}_h(x / \log x)$ is $o(|\mathcal{N}_h(x)|)$. Together, this proves the first part of the theorem. Next, we choose $\epsilon'$ such that $0 < \epsilon' < 1/2$ and choose the function $g(x) = (\log \log x)^{\epsilon'}$. Let $\epsilon > 0$ be arbitrary. Since $\epsilon' < 1/2$, there exists $x_0 \in \mathbb{R}$ such that for all $x \geq x_0$ and for all $n$ with $x/\log x < n \leq x$, 
$$\frac{(\log \log x)^{\epsilon'}}{(\log \log n)^{1/2}} < \frac{(\log \log x)^{\epsilon'}}{(\log \log (x/\log x))^{1/2}} \leq \epsilon.$$ 
Hence as $x \rightarrow \infty$, the numbers $n \in \mathcal{N}_h$ with $x/\log x \leq n \leq x$ that satisfy the inequality
$$\frac{|\omega(n) - \log \log n|}{\log \log n} \geq \epsilon$$
is $o(|\mathcal{N}_h(x)|)$. This together with the fact that the remaining set of natural numbers $n \in \mathcal{N}_h(x / \log x)$ is $o(|\mathcal{N}_h(x)|)$ implies that $\omega(n)$ has normal order $\log \log n$ over $h$-full numbers.   
\end{proof}

In this work, we employ a counting argument to establish that $\omega(n)$ has normal order $\log \log n$ over $h$-free and over $h$-full numbers. In addition, we can also establish that $\omega(n)$ satisfies the Gaussian distribution over the subsets of $h$-free and $h$-full numbers. We will report this result in a follow-up article. The function field analog of this research has been studied by Lal\'in and Zhang \cite{lz}.  

Let $k \geq 1$ be a natural number. Let $\omega_k(n)$ denote the number of distinct prime factors of a natural number $n$ with multiplicity $k$. Elma and Liu \cite{el} studied the distribution of $\omega_k(n)$ over the natural numbers. They established that $\omega_1(n)$ behaves asymptotically similar to $\omega(n)$, whereas $\omega_k(n)$ with $k \geq 2$ is asymptotically smaller compared to $\omega(n)$ over naturals. Moreover, over the natural numbers, $\omega_1(n)$ has normal order $\log \log n$, and $\omega_k(n)$ with $k \geq 2$ does not have a normal order. This naturally raises the question about the behavior of $\omega_k(n)$ over any subset of natural numbers, and if they have the same normal order as $\omega(n)$ over such subsets. We will study the distribution of $\omega_k(n)$ over the particular sets of $h$-free and $h$-full numbers in a separate article and answer the questions related to their behavior. Note that the function field analog of this research has been studied by G\'omez and Lal\'in \cite{lg}.

\section{Acknowledgement}

The authors would like to thank the referees for their valuable comments and for providing an outline of a direct proof of the normal order result over $h$-full numbers using the classical case as discussed in the introduction. The authors would also like to thank Matilde Lal\'in for the helpful discussions.

\bibliographystyle{plain} 

\end{document}